\documentclass[11pt]{amsart}
\usepackage{mathrsfs}
\usepackage{}
\usepackage{amsfonts}
\usepackage{amsfonts,latexsym,rawfonts,amsmath,amssymb,amsthm}
\usepackage[plainpages=false]{hyperref}
\usepackage{fancyhdr}
\usepackage{graphicx}
\makeatletter
\renewcommand\@biblabel[1]{}
\makeatother

\numberwithin{equation}{section}

\newcommand{\beq}{\begin{equation}}
\newcommand{\eeq}{\end{equation}}
\newcommand{\beqs}{\begin{eqnarray*}}
\newcommand{\eeqs}{\end{eqnarray*}}
\newcommand{\beqn}{\begin{eqnarray}}
\newcommand{\eeqn}{\end{eqnarray}}
\newcommand{\beqa}{\begin{array}}
\newcommand{\eeqa}{\end{array}}

\def\lra{\longrightarrow}

\def\bc{\begin{center}}
\def\ec{\end{center}}

\def\begeq{\begin{equation}}
\def\endeq{\end{equation}}
\def\and{\quad{\rm and}\quad}

\let\lra=\longrightarrow

\def\mapright\#1{\,\smash{\mathop{\lra}\limits^{\#1}}\,}

\newtheorem{prop}{Proposition}[section]
\newtheorem{theo}[prop]{Theorem}
\newtheorem{lem}[prop]{Lemma}

\newtheorem{cor}[prop]{Corollary}

\newtheorem{defi}[prop]{Definition}

\begin{document}

\date{}
\author   {Yuxing Deng }
\author { Xiaohua $\text{Zhu}^*$}

\thanks {* Partially supported by the NSFC Grants 11271022 and 11331001}
 \subjclass[2000]{Primary: 53C25; Secondary:  53C55,
 58J05}
\keywords {Ricci soliton, Ricci flow,  pinched Ricci curvature}

\address{ Yuxing Deng\\School of Mathematical Sciences, Peking University,
Beijing, 100871, China}

\address{ Xiaohua Zhu\\School of Mathematical Sciences and BICMR, Peking University,
Beijing, 100871, China\\
 xhzhu@math.pku.edu.cn}

\title{ Complete non-compact  gradient Ricci solitons with nonnegative Ricci curvature }
\maketitle

\section*{\ }

\begin{abstract}  In this paper,  we  give a  delay estimate  of  scalar curvature for  a complete non-compact  expanding  (or steady) gradient Ricci soliton with nonnegative Ricci curvature.  As an application,   we  prove  that  any complete  non-compact expanding (or steady)  gradient K\"{a}hler-Ricci solitons with positively  pinched Ricci curvature should be  Ricci flat.  The result   answers  a  question  in case of  K\"{a}hler-Ricci solitons
proposed by Chow, Lu and Ni   in a book. \end{abstract}

\section {Introduction}

Ricci soliton  plays an  important role in the study of Hamilton's  Ricci flow,  in particular in the singularities analysis of  Ricci flow \cite{H3}, \cite{Ca1}, \cite{Pe}.
In case of shrinking gradient Ricci solitons with positive curvature,  Hamilton  proved  that the solitons  should be isometric to a standard sphere in $\mathbb R^3$  in two dimensional case \cite{H3}. Perelman generalized Hamilton's result to three dimensional case \cite{Pe}.  Later on,  Nabor  proved that any  four-dimensional  shrinking gradient Ricci soliton with positive bounded curvature operator should be a standard  sphere in $\mathbb R^5$ \cite{Na}. On the other hand,  Perelman and Brendle proved that any steady gradient Ricci soliton with nonnegative  sectional curvature should be a Bryant's soliton in case of 2-dimension and 3-dimension, respectively  \cite{Pe}, \cite{CLN},  \cite{Bry}, \cite{Bre}.
However, to author's acknowledge, there is  rarely  understanding  in case of     expanding  gradient Ricci solitons even for lower dimensional manifolds.   For example,     how to  classify complete
 non-compact  gradient   expanding (or steady)  Ricci solitons  under  a suitable curvature
condition.   The purpose of this paper is to   give a rigidity  theorem  for a class of expanding  (or steady) gradient K\"{a}hler-Ricci solitons  with nonnegative Ricci curvature.

\begin{defi}
A complete Riemannian metric $g$ on  $M$ is called a gradient Ricci soliton if there exists a smooth function $f$ ( which is called a defining function)  on $M$ such that
\begin{equation}\label{Def-soliton}
R_{ij}+\rho g_{ij}=\nabla_{i}\nabla_{j}f,
\end{equation}
where $\rho\in \mathbb{R}$ is a constant. The gradient Ricci soliton is called expanding,   steady and shrinking  according to the sign $\rho>, =, <0$,  respectively.
\end{defi}

For simplicity,
we normalize $\rho=1,0,-1$.
  In addition,    $g$  is a  K\"{a}hler  metric on a complex manifold $M$,  we call $g$ is a K\"ahler-Riccoi soliton.   Since  $\overline\partial f$  induces a holomorphic vector field on $M$,     (\ref{Def-soliton})  was  usually written in a complex version,
\begin{equation}
R_{i\bar{j}}+\rho g_{i\bar{j}}=\nabla_{i}\nabla_{\bar{j}}f,
\end{equation}

 A  gradient soliton $(M,g,f)$    is called complete if $g$ and $\nabla f$ are both
complete.  It is known that  the completeness of $(M, g)$ implies the completeness of
$\nabla f$ \cite{Zhang}.  Throughout this paper, we always assume the soliton is complete.  If there is a point $o\in M$ such that $\nabla f(o)=0$,
  we call $o$ an equilibrium point  of $(M,g)$.   By  studying the existence of   equilibrium points,  we prove the boundedness  of scalar curvature  of $g$.

\begin{theo}\label{main-theorem-1}
Let $(M,g)$ be a complete non-compact   expanding gradient Ricci soliton with nonnegative Ricci curvature or a complete non-compact  steady gradient K\"ahler-Ricci soliton with nonnegative bisectional curvature and positive Ricci curvature.   Then  the scalar curvature of $g$  is bounded and  it attains  the maximum at the unique equilibrium point.
\end{theo}

The proof of Theorem  \ref{main-theorem-1} will be given   in   case of expanding  Ricci solitons   in next section.   For   the steady  Ricci solitons,   the  proof for the existence of   equilibrium points
 is a bit different,  although the boundedness of  scalar curvature is  directly   from  an  identity (\ref{scalar-curvature-identity}).   We  will use  a result   of local convergence for K\"ahler-Ricci flow by Chau and Tam   to prove the existence in Section 5 \cite{CT5}.

Theorem  \ref{main-theorem-1} will be applied to prove the following rigidity theorem for  K\"ahler-Ricci solitons  with nonnegative Ricci curvature.

\begin{theo}\label{main-theorem-2}
 Let $(M^n,g)$ be a  complete non-compact gradient K\"{a}hler-Ricci soliton with  non-negative Ricci curvature.  Suppose that there exists a point $p\in M$ such that
  the scalar curvature $R$  of $g$ satisfies
 \begin{align}\label{condition-expending}
  \frac{1}{{\rm vol}(B_{r}(p))}\int_{B_{r}(p)} R {\rm dr} \leq \frac{\varepsilon(r)}{1+r^{2}}, ~\text{if}~ g~\text{ is expending};
  \end{align}
or
 \begin{align}\label{condition-steady} \frac{1}{{\rm vol}(B_{r}(p))}\int_{B_{r}(p)} R {\rm dr} \leq \frac{\varepsilon(r)}{1+r},  ~\text{if}~ g~\text{ is steady},
 \end{align}
where $\varepsilon(r)\rightarrow0$  as $r\to\infty$.  Then  $g$ is Ricci-flat. Moreover, $(M,g)$ is isometric to $\mathbb C^n$ if $g$ is expending.
\end{theo}

We note that   under the condition of   nonnegative sectional curvature (or nonnegative holomorphic bisectional curvature  for K\"ahler manifolds)  several  rigidity theorems  were obtained in \cite{H3},  \cite{N1}, \cite{PT}, etc.  For Ricci solitons,  we are able to use  the Ricci flow to weaken the  condition of curvature to nonnegative Ricci curvature.

As a corollary, we   obtain a version of  Theorem \ref{main-theorem-2} under the  pointed-wise  Ricci decay  condition.

\begin{theo}\label{main-theorem-3} Let $(M^n,g)$ be a complete  non-compact  gradient K\"{a}hler-Ricci soliton with  non-negative Ricci curvature.   Suppose that $g$ satisfies
 \begin{align}\label{condition-expending-2}
  R(x)\leq \frac{\varepsilon(r(x))}{ 1+r(x)^2}~ (\varepsilon(r)\rightarrow0, \text{ as}~ r\to\infty) ,  ~\text{if}~ g~\text{ is expending};
   \end{align}
 or
  \begin{align}\label{condition-steady-2}
     R(x)\leq \frac{C}{1+r(x)^{2n+\epsilon}}~\text{for some}~C, \epsilon>0, ~\text{if}~ g~\text{ is steady}.
   \end{align}
Then  $g$   is Ricci-flat. Moreover, $(M,g)$ is isometric to $\mathbb C^n$ if $g$ is expending.

\end{theo}

In case of  steady  solitons in  Theorem  \ref{main-theorem-3},  if  we  assume that $(M,g)$ has nonnegative bisectional curvature instead of  nonnegative Ricci curvature, then  the  condition (\ref{condition-steady-2}) can be  weakened as
\begin{align}\label{strong-condition-steady-2}
 R(x)\leq \frac{C}{1+r(x)^{1+\epsilon}}.
\end{align}
In fact, we can prove that  $(M,g)$ is isometric to $\mathbb C^n$
by Theorem  \ref{main-theorem-1},  see  Proposition  \ref{proposition-5.7}.
Proposition  \ref{proposition-5.7}  is an analogy of Hamilton's result  for K\"ahler  manifolds \cite{H1}.

A  Riemannian metric is called  with property of positively  pinched Ricci curvature if  there is a uniform constant $\delta>0$ such that $\text{Ric}(g)\ge \delta R g$ \cite{H1}, \cite{H3}.
It was proved that the scalar curvature of complete  non-compact expanding (or steady) gradient Ricci solitons with positively  pinched Ricci curvature has exponential decay (cf.  Theorem 9.56, \cite{CLN}).  Thus as a direct consequence of Theorem \ref{main-theorem-2},  we obtain

\begin{cor}\label{corollary-1}
Non-trivial  complete  non-compact expanding or steady gradient K\"{a}hler-Ricci soliton with positively pinched Ricci curvature doesn't exist for $n\ge 2$.
\end{cor}

 Corollary  \ref{corollary-1} answers  a question   in case of K\"ahler-Ricci solitons proposed by Chow, Lu and Ni in their book   \cite{CLN}  (cf.  page 390).   They asked whether there exists an expanding gradient Ricci soliton with positively pinched Ricci curvature when $n\ge 3$.

Theorem \ref{main-theorem-2} and \ref{main-theorem-3} will be proved  in Section \ref{section-3} and Section \ref{section-4} according to
expending or steady solitons, respectively.

\section{ Boundedness of scalar curvature--I }

In this section,  we  prove the boundedness  of scalar curvature in case of expending Ricci solitons.  Let $(M^{n},g)$ be a Riemannian manifold. In local coordinates $(x^{1},x^{2},\cdots,x^{n})$,   curvature tensor $\text{Rm}$  of $g$ is defined by
\begin{equation}
{\rm Rm}(\frac{\partial}{\partial x^{i}}, \frac{\partial}{\partial x^{j}})\frac{\partial}{\partial x^{k}}\triangleq  \sum R^{l}_{ijk}\frac{\partial}{\partial x^{l}}\notag
\end{equation}
and $R_{ijkl}\triangleq  \sum g_{lm}R_{ijk}^{m}$.  Then the Ricci curvature is given by
 $$R_{jk}=\sum R_{ijk}^{i}.$$
 Thus by  the commutation  formula,
\begin{align}\label{commutation-formula}
(\nabla_{i}\nabla_{j}-\nabla_{j}\nabla_{i})\alpha_{k_{1}\cdots k_{r}}=-\sum_{l=1}^{r}R^{m}_{ijk_{l}}\alpha_{k_{1}\cdots k_{l-1}mk_{l+1}\cdots k_{r}},
\end{align}
we get from  the Bianchi identity,
\begin{equation}\label{contacted-Bianchi}
2\sum \nabla_{i}R_{ij}=\nabla_{j}R.
\end{equation}

 Let $(M^{n},g,f)$ be an expanding gradient Ricci soliton and  $\phi_{t}$ be a family of  diffeomorphisms  generated by $-\nabla f$.   Then the  induced metrics
  $g(t)=\phi_{t}^{*}g$   satisfy
\begin{equation}\label{normalized-ricci-flow}
\frac{\partial}{\partial t}g=-2\text{Ric}(g)-2g.
\end{equation}
(\ref{normalized-ricci-flow}) is equivalent to
\begin{equation}\label{expanding-soliton}
R_{ij}(t)+g_{ij}(t)=\nabla_{i}\nabla_{j}f(t),
\end{equation}
where $f(t)=\phi_{t}^{*}f$ and  $\nabla$ is taken w.r.t $g(t)$.

\begin{lem}\label{Indentity}
$$\frac{\partial}{\partial t}R=2{\rm Ric}(\nabla f(t),\nabla f(t)).$$
\end{lem}

\begin{proof}
 Differentiating  $(\ref{expanding-soliton})$  on both sides,  we have
\begin{align}
\nabla_{k}R_{ij}=\nabla_{k}\nabla_{i}\nabla_{j}f. \notag
\end{align}
It follows from (\ref{commutation-formula}),
\begin{align}
\nabla_{i}R_{jk}-\nabla_{j}R_{ik}=-\sum R_{ijkl}\nabla_{l}f.\notag
\end{align}
Thus by (\ref{contacted-Bianchi}),  we get
\begin{align}\label{L1-1}
\nabla_{j}R=-2R_{jl}\nabla_{l}f.
\end{align}
Hence
\begin{align}
\frac{\partial}{\partial t}R(x,t)=\frac{\partial}{\partial t}R(\phi_{t}(x),0)=-\langle \nabla R ,\nabla f\rangle =2\text{Ric}(\nabla f,\nabla f).\notag
\end{align}

\end{proof}

 Let $B_{r}(o,t)$ be a $r$-geodesic ball  centered at $o\in M$  w.r.t $g(t)$.  Then

\begin{lem}\label{domain-constract}   Let  $g(x,t)$ be a solution of (\ref{normalized-ricci-flow}) with nonnegative Ricci curvature for any $t\in (0,\infty)$.
 Then for any $r>0$ and $\delta>0$, there exists a $T_{0}=T_0(r,\delta)>0$ such that $B_{r}(o,0)\subset B_{\delta}(o,t)$ for any $t\geq T_{0}$.
\end{lem}

\begin{proof}   By  (\ref{normalized-ricci-flow}),  it is easy to see that
\begin{align}
\frac{ {\rm d} |v|^2_t}{{\rm d} t}\leq-2|v|^{2}_{t} ,~\forall~ t\geq0,\notag
\end{align}
 where $v\in T_{p}^{(1,0)}M$ for any $p\in M$.  Then
  $$
|v|_{t}^{2}\leq  e^{-2t}|v|_{0}^{2}.$$
Connecting  $o$ and $p$ by  a minimal geodesic curve   $\gamma(s)$  with  an arc-parameter $s$  w.r.t the metric $g(x,0)$   in  $B_{r}(o,0)$,    we get
\begin{equation}\label{vector-compare}
 d_{t}(o,p)\leq\int_{0}^{l}|\gamma^{\prime}(s)|_{t}{\rm ds}\leq\int_{0}^{l}|\gamma^{\prime}(s)|_{0} e^{-2t}{\rm ds}\leq  r e^{-2t},
\end{equation}
where $l$ is the length of $\gamma(s)$.  Therefore,  by taking  $t$ large enough.  we see that $B_{r}(o,1)\subset B_{\delta}(o,t)$.
\end{proof}

Taking  an integration along a geodesic curve on both sides of (\ref{Def-soliton}),   on can show that $f(x)\ge \frac{r(x)^{2}}{4}$ under the assumption of nonnegative
Ricci curvature.   This implies that $f(x)$ attains the  minimum  at some point $o\in M$.  Thus $\nabla f(o)=0$.   Moreover the equilibrium point  $o$ is unique.  This is because, if there is another  equilibrium point $p$,  then $\phi_{t}(o)=o$ and $\phi_{t}(p)=p$. In particular $d_{t}(o,p)=d_{0}(o,p)$.  On the other hand, by
(\ref{vector-compare}),   $d_{t}(o,p)\le e^{-2t}d_{0}(o,p),~\forall ~t>0$.  Hence,  $d_{0}(o,p)=0$, and consequently   $o=p$.

Now we begin to prove Theorem \ref{main-theorem-1}.

\begin{proof}[Proof of Theorem \ref{main-theorem-1}~(the expanding case)]  Let  $o$ be  the unique equilibrium point.  Then by Lemma \ref{domain-constract},  for any $r>\delta>0$, there exists $T_0$ such that
$$B_{r}(o,0)\subset B_{\delta}(o,t),\mbox{\quad }\forall\mbox{\ } t\ge T_{0}.$$
On the other hand, by Lemma \ref{Indentity},  we  see that   $R(x,t)$ is nondecreasing in $t$.  Thus
\begin{equation}\label{T1-1}
\sup_{x\in B_{r}(o,0)}R(x,0)\le\sup_{x\in B_{\delta}(o,t)}R(x,t),\mbox{\quad }\forall\mbox{\ }r>0, \mbox{\ } \delta>0.
\end{equation}
Note that $\phi_{t}:(M^{n},g(t))\rightarrow(M^{n},g(0))$  are a family of  isometric deformations.  It follows
\begin{equation}\label{T1-2}
\sup_{x\in B_{\delta}(o,0)}R(x,0)=\sup_{x\in B_{\delta}(o,t)}R(x,t),\mbox{\quad }\forall \mbox{\ } \delta>0.
\end{equation}
 Hence, combining  $(\ref{T1-1})$ and $(\ref{T1-2})$, we get
 $$\sup_{x\in B_{r}(o,0)}R(x,0)=\sup_{x\in B_{\delta}(o,0)}R(x,0),\mbox{\quad }\forall\mbox{\ }r>\delta>0.$$
 Let $\delta\rightarrow0$,  we derive
$$\sup_{x\in B_{r}(o,0)}R(x,0)=R(o,0),\mbox{\quad }\forall\mbox{\ }r>0.$$
This proves the theorem.
\end{proof}

\section{Expanding  K\"ahler-Ricci solitons}\label{section-2}

In this section, we prove both  Theorem \ref{main-theorem-2} and  Theorem \ref{main-theorem-3} in case of expanding  K\"ahler-Ricci solitons.  Theorem \ref{main-theorem-3} is  a consequence  of  Theorem \ref{main-theorem-2} by the following lemma.

\begin{lem}\label{average-decay-1}
Let $(M, g)$ be an expanding gradient K\"{a}hler-Ricci soliton which  satisfies (\ref{condition-expending-2}) in Theorem \ref{main-theorem-3}.  Then there exists a function
$\varepsilon'(r)$ ( $\varepsilon'(r)\to 0$ as $r\to \infty$)  such that
\begin{equation*}
\frac{1}{{\rm vol}(B_{r}(p))}\int_{B_{r}(p)}R{\rm dv}\leq\frac{\varepsilon'(r)}{1+r^{2}}.
\end{equation*}
\end{lem}

\begin{proof}
Note that  an expanding Ricci soliton with nonnegative Ricci curvature  has maximal volume growth (cf.  \cite{CN} or \cite{CT4}). Namely,  there exists a  uniform constant  $\delta>0$ such that
$$\text{vol}(B_{r}(p))\ge \delta r^{2n}.$$
  On the other hand,  by the volume comparison theorem,  we have
$$\text{vol}(\partial B_{r}(p))\leq n\frac{\text{vol}( B_{r}(p))}{r}\leq Cr^{2n-1},$$
  where $C$ is a uniform constant.  Thus
\begin{align*}
\frac{1}{\text{vol}(B_{r}(p))}\int_{B_{r}(p)}R{\rm dv}=&\frac{1}{\text{vol}(B_{r}(p))}\int_{0}^{r}{\rm ds}\int_{\partial B_{s}(p)}R{\rm d\sigma}\\
\leq& \frac{1}{ \delta r^{2n}}\int_{0}^{r}\frac{\varepsilon(s)}{1+s^{2}}\text{vol}(\partial B_{s}(p)){\rm ds} \\
\leq& \frac{1}{ \delta r^{2n}}\int_{0}^{r}C(s+1)^{2n-3}\varepsilon(s){\rm ds} \\
\leq& \frac{\varepsilon'(r)}{1+r^{2}},
\end{align*}
where $\varepsilon'(r)\rightarrow0$  as $r\rightarrow0$.
\end{proof}

\begin{proof}[Proof of Theorem \ref{main-theorem-2}~(the expanding case)]
\textbf{Ricci-Flat:} Let $\phi_{t}$ be  a family of  diffeomorphisms  generated by $-\nabla f$.  Let $g(t)=\phi_{t}^{*}g$ and
  $\widehat{g}(\cdot, t)=t g(\cdot,\ln t)$. Then  $\widehat{g}(\cdot, t)$ satisfies
\begin{equation}\label{kr-flow}
\left\{ \begin{aligned}
         \frac{\partial}{\partial t}\widehat{g}_{i\bar{j}}(x,t) &= -\widehat{R}_{i\bar{j}}(x,t) \\
                  \widehat{g}_{i\bar{j}}(x,1)&=g_{i\bar{j}}(x).
                          \end{aligned} \right.
                          \end{equation}

   Let $F(x,t)=\ln\det(\widehat{g}_{i\bar{j}}(x,t))-\ln\det(\widehat{g}_{i\bar{j}}(x,1))$. By $(\ref{kr-flow})$, it is easy to see
\begin{align}
F(x,t)=-\int_{1}^{t}\widehat{R}(x,s)\rm{ds}\leq 0.\notag
\end{align}
 Since
$$t\widehat{R}(o,t)=R(o,\ln t)=R(o,0),~\text{and}~t\widehat{R}(x,t)=R(x,\ln t)\leq R(o,0),$$
 where  $o$ is  the equilibrium point of $M$,   by Theorem \ref{main-theorem-1},   $F$  is uniformly bounded  on $x$.  Moreover,   we have
\begin{align}\label{infty}
 M(t)\doteq-\inf_{x\in M}F(x,t)=R(o,0)\ln t.
  \end{align}
  In the following,  we  shall  estimate the upper bound of $M(t)$ by using the Green integration  as in \cite{S3} (also see \cite{N}).

  By  a direct computation,  we have
\begin{equation}\label{T2-1}
\Delta_{1}F(x,t)=\widehat{R}(x,1)-g^{i\bar{j}}(x,1)\widehat{R}_{i\bar{j}}(x,t)\leq \widehat{R}(x,1)+\frac{\partial}{\partial t} e^{F(x,t)},
\end{equation}
where  the Lapalace $\Delta_{1}$ is w.r.t $\widehat{g}(x,1)$.
 Let $G_{r}(x,y)$ be a positive Green's function  with zero boundary value   w.r.t $\widehat{g}(x,1)$  on $\hat B_{r}(x,1)$.  Note that
 $$\int_{\hat B_{r}(x_{0},1)}\frac{\partial G_{r}(x_{0},y)}{\partial \nu}{\rm ds}=-1\text{ and}~ \frac{\partial G_{r}(x_{0},y)}{\partial \nu}\leq0.$$
 By  integrating (\ref{T2-1}) on both sides,  we   have
\begin{align}\label{inequality-1}
&\int_{\hat B_{r}(x_{0},1)}G_{r}(x_{0},y)(1-e^{F(x,t)}){\rm dv}\notag\\
\leq&\int_{1}^{t}{\rm ds} \int_{\hat B_{r}(x_{0},1)}G_{r}(x_{0},y)(-\Delta_{1}F(x,s)){\rm dv}\notag\\
&+t\int_{\hat B_{r}(x_{0},1)}G_{r}(x_{0},y)\widehat{R}(y,1){\rm dv}\notag\\
=&\int_{1}^{t}\Big(F(x_{0},s)+\int_{\hat B_{r}(x_{0},1)}\frac{\partial G_{r}(x_{0},y)}{\partial \nu}F(x,s){\rm dv}\Big){\rm ds}\notag\\
&+t\int_{\hat B_{r}(x_{0},1)}G_{r}(x_{0},y)\widehat{R}(y,1){\rm dv}\notag\\
\leq&t\Big(M(t)+\int_{\hat B_{r}(x_{0},1)}G_{r}(x_{0},y)\widehat{R}(y,1){\rm dv}\Big).
\end{align}
On the other hand, by  the Green function estimate (cf.  Lemma 1.1 in \cite{T}),
$$G_{r}(x,y)\geq C_{1}^{-1}\int_{d(x,y)}^{r}\frac{s}{\text{vol}(\hat B_{s}(x,1))}{\rm ds},\mbox{\quad}\forall \mbox{\ }y\in \hat B_{\frac{r}{5}}(x,1),$$
where $C_1$  is a uniform constant,  we get
\begin{align*}
&\int_{\hat B_{r}(x_{0},1)}G_{r}(x_{0},y)(1-e^{F(x,t)}){\rm dv}\\
\geq & C_1^{-1}\int_{0}^{\frac{r}{5}}\Big(\int_{\tau}^{r}\frac{s}{\text{vol}(\hat B_{s}(x_{0},1))}{\rm ds}\Big)\Big(\int_{\partial \hat B_{\tau}(x_{0},1)}(1-e^{F(x,t)}){\rm d\sigma}\Big){\rm d\tau}\\
\geq & C_1^{-1}\int_{0}^{\frac{r}{5}}\Big(\int_{\frac{r}{5}}^{r}\frac{s}{\text{vol}(\hat B_{s}(x_{0},1))}{\rm ds}\Big)\Big(\int_{\partial \hat B_{\tau}(x_{0},1)}(1-e^{F(x,t)}){\rm d\sigma}\Big){\rm d\tau}
\\
\geq& \frac{C_2^{-1}r^{2}}{\text{vol}(\hat B_{\frac{r}{5}}(x_{0},1))}\int_{\hat B_{\frac{r}{5}}(x_{0},1)}(1-e^{F(x,t)}){\rm dv}\\
\geq& \frac{eC_2^{-1}r^{2}}{\text{vol}(\hat B_{\frac{r}{5}}(x_{0},1))}\int_{\hat B_{\frac{r}{5}}(x_{0},1)}\frac{-F(x,t)}{1-F(x,t)}{\rm dv}.
\end{align*}
It follows
\begin{align}
&\frac{r^{2}}{\text{vol}(\hat B_{\frac{r}{5}}(x_{0},1))}\int_{\hat B_{\frac{r}{5}}(x_{0},1)}(-F(x,t)){\rm dv}\notag\\
&\leq C_{3}(1+M(t)) \int_{\hat B_{r}(x_{0},1)}G_{r}(x_{0},y)(1-e^{F(x,t)}){\rm dv}\notag
\end{align}
 Hence by (\ref{inequality-1}), we derive
 \begin{align}\label{inequality-2}
&\frac{r^{2}}{\text{vol}(\hat B_{\frac{r}{5}}(x_{0},1))}\int_{\hat B_{\frac{r}{5}}(x_{0},1)}(-F(x,t)){\rm dv}\notag\\
&\leq  C_3t(1+M(t))\Big(M(t)+\int_{\hat B_{r}(x_{0},1)}G_{r}(x_{0},y)\widehat{R}(y,1){\rm dv}\Big).
\end{align}

By $(\ref{T2-1})$,  we have $\Delta_{1}(-F(x,t))\ge -\hat R(x,1)$. Then  by  the mean value inequality (cf. Lemma 2.1 of \cite{N}),   we see
 \begin{align}\label{inequality-3}
&-F(x_{0},t)\leq \frac{C(n)}{\text{vol}(\hat B_{\frac{r}{5}}(x_{0},1))}\int_{\hat B_{\frac{r}{5}}(x_{0},1)}(-F(x,t)){\rm dv}\notag\\
&+\int_{\hat B_{\frac{r}{5}}(x_{0},1)}G_{\frac{r}{5}}(x_{0},y)\widehat{R}(y,1){\rm dv}.
 \end{align}
Hence,  to get  an  upper bound of  $ -F(x_{0},t)$,   by $(\ref{inequality-2})$ and $(\ref{inequality-3})$,
we shall estimate the integral  $\int_{\hat B_{\frac{r}{5}}(x_{0},1)}G_{\frac{r}{5}}(x_{0},y)\widehat{R}(y,1)\rm{dv}$.

 Recall  the Li-Yau's estimate for the Green function:  There exits a positive Green's function $G(x, y)$ such that (cf. Theorem 5.2 in \cite{LY})
\begin{align}\label{li-yau} C(n)^{-1}\int_{d^{2}(x,y)}^{+\infty}\frac{{\rm dt}}{\text{vol}(\hat B_{\sqrt{t}}(x))}\leq G(x,y)
\leq C(n)\int_{d^{2}(x,y)}^{+\infty}\frac{{\rm dt}}{\text{vol}(\hat B_{\sqrt{t}}(x))}.
\end{align}
Then
\begin{align}\label{green-integral-2}
&\int_{\hat B_{r}(x_{0},1)}G_{r}(x_{0},y)\widehat{R}(y,1){\rm dv}\leq \int_{0}^{r}{\rm ds}\int_{\partial \hat B_{s}(x_{0},1)}G(x_{0},y)\widehat{R}(y,1){\rm d\sigma}\notag\\
\leq& C(n)\int_{0}^{r}{\rm ds} \Big(\int_{\partial \hat B_{s}(x_{0},1)}\widehat{R}(y,1){\rm d\sigma}\int_{d^{2}(x,y)}^{+\infty}\frac{{\rm dt}}{\text{vol}(\hat B_{\sqrt{t}}(x))}\Big)\notag\\
=& C(n)\Big( \int_{r^{2}}^{+\infty}\frac{{\rm dt}}{\text{vol}(\hat B_{\sqrt{t}}(x))}\int_{\hat B_{r}(x_{0},1)}\widehat{R}{\rm dv}\notag\\
+&\int_{0}^{r}\frac{s}{\text{vol}(\hat B_{s}(x_{0},1))}\int_{\hat B_{s}(x_{0},1)}\widehat{R}{\rm dvds}\Big),\mbox{\quad}\forall \mbox{\ }r>0.
\end{align}
 Since $(M,g)$ has the maximal volume growth,   there exists a uniform constant $\delta>0$ such that
 $$\frac{\text{vol}(B_{s}(x))}{\text{vol}(B_{t}(x))}\geq \delta\Big(\frac{s}{t}\Big)^{2n}, ~\forall~ s\geq t\geq c_{0}.$$
It follows
 \begin{align}
  \int_{d^{2}(x,y)}^{+\infty}\frac{{\rm dt}}{\text{vol}(B_{\sqrt{t}}(x))}\leq C_{4}\frac{d^{2}(x,y)}{\text{vol}(B_{d(x,y)}(x))}\mbox{,\quad}\forall\mbox{\ } d(x,y)\geq c_{0}.\notag
 \end{align}
 Hence, we get from (\ref{green-integral-2}),
\begin{align}\label{green-integral-3}
&\int_{\hat B_{r}(x_{0},1)}G_{r}(x_{0},y)\widehat{R}(y,1){\rm dv}\notag\\
&\leq C(n)\Big( C_{4}\frac{r^{2}}{\text{vol}(\hat B_{r}(x_{0},1)}\int_{\hat B_{r}(x_{0},1)}\widehat{R}{\rm dv}\notag\\
&+\int_{0}^{r}\frac{s}{\text{vol}(\hat B_{s}(x_{0},1)}\int_{\hat B_{s}(x_{0},1)}\widehat{R}{\rm dvds}\Big),\mbox{\quad}\forall \mbox{\ }r>0.
\end{align}

By  the volume comparison  theorem and the condition (\ref{condition-expending-2}),  we have
\begin{align*}
\frac{1}{\text{vol}(B_{r}(o))}\int_{B_{r}(o)}R(x){\rm dv}\leq& \frac{\text{vol}(B_{r+d}(p))}{\text{vol}(B_{r}(o))}\cdot \frac{1}{\text{vol}(B_{r+d}(p))}\int_{B_{r+d}(p)}R(x)\rm{dv}\\
\leq& \frac{\text{vol}(B_{r+2d}(o))}{\text{vol}(B_{r}(o))}\cdot\frac{\varepsilon(r+d)}{1+(r+d)^{2}}\\
\leq& \Big(\frac{r+2d}{r}\Big)^{2n}\frac{\varepsilon(r+d)}{1+(r+d)^{2}},
\end{align*}
where $d=d(o,p)$.  Then there exists another function  $\varepsilon'(r) $ ($\varepsilon'(r)\to 0$ as $r\to \infty$) such that
\begin{equation}
\frac{1}{\text{vol}(B_{r}(o))}\int_{B_{r}(o)}R(x){\rm dv}\leq \frac{\varepsilon'(r)}{1+r^{2}}.\notag
\end{equation}
Thus inserting the above inequality into (\ref{green-integral-3}),  we derive at  $x_{0}=o$,
\begin{equation}\label{inequality-5}
\int_{\hat B_{r}(o,1)}G_{r}(o,y)\widehat{R}(y,1){\rm dv}\leq \varepsilon''(r)+\varepsilon''(r)\ln(1+r^{2}),
\end{equation}
where $\varepsilon''(r)\to 0$ as $r\to \infty$.

  Combining $(\ref{inequality-1})$, $(\ref{inequality-2})$ and $(\ref{inequality-5})$,   it is easy to see
\begin{align*}
-F(o,t)\leq r^{-2}C(n)t(M(t)+1)\Big(M(t)+C_5\varepsilon''(r)+C_6\varepsilon''(r)\ln(1+r^{2})\Big)\\+ \Big(C_7\varepsilon''(r)+C_{8}\varepsilon''(r)\ln(1+r^{2})\Big),\mbox{\quad}\forall \mbox{\ }r>0.
 \end{align*}
Note that $-F(o,t)=M(t)=R(o,0)\ln t$.  Then by taking $r=t$, we obtain
\begin{align*}
R(o,0)\ln t\leq& C_9\varepsilon''(t)+C_{10}\varepsilon''(t)\ln(1+t)+C_{11}\frac{\ln t}{t},\mbox{\quad}\forall \mbox{\ }t\ge1.
 \end{align*}
Dividing  by $\ln t$  on  both sides of  the above inequality and letting $t\rightarrow\infty$,  we  deduce  $R(o,0)=0$. Hence we prove that
$g$ is Ricci flat

\textbf{Flatness:} We shall  further prove that $g$ is a flat metric on $\mathbb C^n$. Note that
 \begin{align} \label{hessian}g=\text{hess} f \end{align}
since $g$ is Ricci flat. Then $f$ is strictly convex and $f$ attains the minimum at $o$.
By a direct computation, we have
\begin{align}
\langle \nabla f,X\rangle=(\nabla {\rm d}f)(\nabla f,X)=\langle \nabla_{\nabla f} \nabla f, X\rangle, \mbox{\ }\forall \mbox{\ }X\in\Gamma^{\infty}(TM).\notag
\end{align}
It follows
$$\nabla_{\nabla f} \nabla f=\nabla f.$$
Thus
\begin{equation}\nabla_{\frac{\nabla f}{|\nabla f|}}\Big(\frac{\nabla f}{|\nabla f|}\Big)=\frac{1}{|\nabla f|}(\frac{\nabla_{\nabla f} \nabla f}{|\nabla f|}-\frac{\langle\nabla_{\nabla f} \nabla f,\nabla f\rangle}{|\nabla f|^3}\nabla f)=0,~x\in M\setminus \{o\}. \notag
\end{equation}
This implies that any integral curve generated by $\frac{\nabla f}{|\nabla f|}$ ($x\in M\setminus \{o\}$) is geodesic.

Let $\phi_{t}$ and $\varphi_{t}$ be  one-parameter diffeomorphisms groups generated by $-\nabla f$ and $-\frac{\nabla f}{|\nabla f|}$, respectively. Then as in the proof of (\ref{vector-compare}), we  have $d(\phi_{t}(x),o)=e^{-t}d(x,o)$. Thus $\langle\nabla f,\nabla r\rangle=-\frac{{\rm d}}{{\rm dt}}d(\phi_{t}(x),o)>0$. This shows that $\varphi_{s}(x)$ is a geodesic curve from $x$ to $o$.
 Let $\gamma(s)=\varphi_{d(x,o)-s}(x)$. Then $\gamma(s)$ is a minimal geodesic curve from $o$ to $x$  as long as $\text{dist}(o,x)\le r_0<<1.$ Moreover, we have
\begin{align}
\left\{\begin{aligned}
  &\frac{{\rm d^{2}}}{{\rm ds^{2}}}f(\gamma(s)) = 1,\notag \\
                  &\frac{{\rm d}}{{\rm ds}}f(\gamma(s)) = |\nabla f|\rightarrow0,\mbox{\ as\ } s\rightarrow0,\notag \\
                &  f(\gamma(0))=f(o)=0.
                          \end{aligned} \right.\notag
                          \end{align}
Therefore, we deduce
\begin{align} f(x)=\frac{1}{2}r^{2}(x),~ \text{if}~ r(x)\leq r_{0}.\notag\end{align}
In particular,
\begin{align} \label{f-function} |\nabla  f(x)|=r.\end{align}

 We claim that $g$ is flat on $B_{r_{0}}(o)$.  Since  $\partial B_{r_{0}}(o)$ is diffeomorphic to $\mathbb{S}^{2n-1}$, we can  choose an orthonormal basis $\{e_{1},\cdots,e_{2n-1}\}$ on $\partial B_{r_{0}}(o)$.  Let $X_{i}(\varphi_{t}(x))=(\varphi_{t})_{*}e_{i}$ for $x\in \partial B_{r_{0}}(o)$,  $1\leq i\leq 2n-1$. Then $\{\nabla r=\frac{\nabla f}{|\nabla f|},X_{1},\cdots,X_{2n-1}\}$ is a global frame on $B_{r_{0}}(o)\setminus \{o\}$. Clearly,  $\mathscr{L}_{\nabla r}X_{i}=0$, $1\leq i\leq 2n-1$.
  Thus by (\ref{f-function}) and (\ref{hessian}),  it follows
\begin{align*}
\frac{\partial}{\partial r}\langle\nabla r,X_{i}\rangle=&\mathscr{L}_{\nabla r}\langle\nabla r,X_{i}\rangle\\
=&(\mathscr{L}_{\nabla r}g)(\nabla r, X_{i})+\langle\mathscr{L}_{\nabla r}\nabla r,X_{i}\rangle+\langle\nabla r, \mathscr{L}_{\nabla r}X_{i}\rangle\\
=&\frac{2}{r}{\rm Hess}f(\nabla r,X_{i})\\
=&\frac{2}{r}\langle\nabla r,X_{i}\rangle.
\end{align*}
Since $\langle\nabla r,X_{i}\rangle|_{r=r_{0}}=0$, we get $\langle\nabla r,X_{i}\rangle=0$ for any $x\in B_{r_{0}}(o)\setminus \{o\}$.
Similarly, we have
\begin{equation}
\left\{ \begin{aligned}
        \frac{\partial}{\partial r}\langle X_{i}, X_{j}\rangle=& \frac{2}{r}\langle X_{i}, X_{j}\rangle,\\
                 \langle\nabla r,X_{i}\rangle|_{r=r_{0}}&=\delta_{ij}.
                          \end{aligned} \right.
                          \end{equation}
Consequently, $\langle X_{i}, X_{j}\rangle=\frac{r^{2}}{r_{0}^{2}}\delta_{ij}$ for any $x\in B_{r_{0}}(o)\setminus \{o\}$.  Hence,
$$g={\rm dr}\otimes{\rm dr}+\frac{r^{2}}{r_{0}^{2}}\sum_{i,j=1}^{2n-1}{\rm d\theta^{i}}\otimes {\rm d\theta^{j}},$$
where  $\{dr,\theta^{1},\cdots,\theta^{2n-1}\}$ are  the corresponding coframe of  $\{\nabla r=\frac{\nabla f}{|\nabla f|},X_{1},$
\newline $\cdots,X_{2n-1}\}$. This proves that  $g$ is  isometric to an Euclidean metric on $B_{r_0}(o)$. Therefore,  $g$ is flat on $B_{r_0}(o)$. The claim is true.

At last, we show that $g$ is globally flat.  Since $\phi_{t}$ is an isometric diffeomorphism  from $(B_{r_{0}}(o,t),g(x,t))$ to $(B_{r_{0}}(o,0),g(x,0))$.
We see that $(B_{r_{0}}(o,t),g(x,t))$ is flat by the above claim. On the other hand,  by  the flow $(\ref{normalized-ricci-flow})$ and  the fact that $g$ is Ricci-flat,
we have $g(x,t)=e^{-2t}g(x,0)$. Hence, $(B_{r_{0}}(o,t),g(x,0))$ is also flat.  Since $M$ is exhausted by $B_{r_{0}}(o,t)$ as $t\rightarrow\infty$  according to Lemma \ref{domain-constract},  we  see that $g$ is globally flat and $M$ is simply connected.   As a consequence,  $(M,g)$ is  isometric to $\mathbb C^n$.
\end{proof}

\section{Steady K\"ahler-Ricci solitons}\label{section-3}

In this section, we deal with  steady gradient Ricci solitons $(M^{n},g,f)$.  As in Section 3,   we let $\phi_{t}$ be a family of  diffeomorphisms generated by $-\nabla f$ and  $g(\cdot, t)=\phi_{t}^{*}g$.  Then $g(\cdot, t)$ satisfies
\begin{equation}\label{ricci-flow-2}
\frac{\partial}{\partial t}g_{ij}=-2R_{ij}.
\end{equation}
It turns
\begin{equation*}
R_{ij}(t)=\nabla_{i}\nabla_{j}f(t),
\end{equation*}
where   $f(t)=\phi_{t}^{*}f$ and $\nabla$ is taken w.r.t $g(t)$.
Hence by the   Bianchi identity (\ref{contacted-Bianchi}), one can obtain
\begin{equation}\label{scalar-curvature-identity}
R+|\nabla f|^{2}=\text{const}.
\end{equation}
This shows that the scalar curvature of $g$  is  uniformly bounded.

Analogous to Lemma  \ref{Indentity},  we  have

\begin{lem}\label{Indentity-2}
$\frac{\partial}{\partial t}R=2{\rm Ric}(\nabla f(t),\nabla f(t))$.
\end{lem}

In general,  we do not know  whether a steady gradient Ricci soliton admits an equilibrium point.  However,
 we can still  prove  Rigidity Theorem \ref{main-theorem-2} in case of steady gradient K\"{a}hler-Ricci solitons  by using the fact of boundedness of  scalar curvature.

\begin{proof}[Proof of Theorem \ref{main-theorem-2}~(the steady case)] Since $ (M,g)$ is  K\"ahlerian,  we  may  rewrite (\ref{ricci-flow-2})  as,
\begin{equation}\label{kr-flow-2}
\left\{ \begin{aligned}
         \frac{\partial}{\partial t}g_{i\bar{j}}(x,t) &= -R_{i\bar{j}}(x,t) \\
                  g_{i\bar{j}}(x,0)&=g_{i\bar{j}}(x).
                          \end{aligned} \right.
                          \end{equation}
In order to get  the estimate for  the Green function as in Section 3,  we use a trick in  \cite{S3} to consider a product space   $\widehat{M}=M\times \mathbb{C}^{2}$  with  a product metric $\widehat{g}=g+dw^{1}\wedge d\bar{w}^{1}+dw^{2}\wedge d\bar{w}^{2}$. Then   $\widehat{g}(x,t)=g(x,t)+dw^{1}\wedge d\bar{w}^{1}+dw^{2}\wedge d\bar{w}^{2}$ is a solution of  (\ref{kr-flow-2}) on  $\widehat{M}$ with the initial metric $\widehat{g}$.

 It was proved  by Shi   that  for any $s>t$ and $B_{t}(x)\subset B_{s}(x)\subset (\widehat M, \widehat g)$ (cf. Section 6 in \cite{S3}),  it holds
 $$\frac{\text{vol}(B_{s}(x))}{\text{vol}(B_{t}(x))}\geq C_{0}^{-1}\Big(\frac{s}{t}\Big)^{4}.$$
 Then
 \begin{equation*}\int_{d^{2}(x,y)}^{+\infty}\frac{{\rm dt}}{\text{vol}(B_{\sqrt{t}}(x))}\leq C_{0}\frac{d^{2}(x,y)}{\text{vol}(B_{d(x,y)}(x))}\mbox{,\quad}\forall\mbox{\ } d(x,y)\geq c_{0}.
 \end{equation*}
  Thus by the Li-Yau estimate \cite{LY},  there exists  a  global  Green's function  $G$ on $(\widehat{M},\widehat{g})$ which satisfies (\ref{li-yau}).

   Let $F(x,t)=\ln\det(\widehat{g}_{i\bar{j}}(x,t))-\ln\det(\widehat{g}_{i\bar{j}}(x,0))$.  By (\ref{kr-flow-2}), it is easy to see
      \begin{equation}
F(x,t)=-\int_{0}^{t}\widehat{R}(x,s){\rm ds}.
\end{equation}
Then by Lemma  \ref{Indentity-2}, we have
  $$  t\widehat  R(x, 0)\le -F(x,t)\leq   C_0t,$$
  where  $C_0=\sup_{\widehat M} \widehat R(x,t).$  Thus as in Section 3,
 to prove that $R\equiv 0$,  we  shall  give a growth estimate  of $-F(x,t)$ on $t$.

Fix an arbitrary point $x_{0}\in M$. For convenience, we denote  a $r$-geodesic ball  $B_{r}(x_{0},t)$ of $(\widehat{M},\widehat{g}(t))$ centered at $(x_{0},0,0)$.  As in Section 3,  by using the Green formula, we can estimate
 \begin{align}\label{green-formula-2}-F(x_{0},t)\leq \frac{C_{1}t(1+M(t))}{r^{2}}&\Big(M(t)+\int_{B_{r}(x_{0},0)}G(x_{0},y)\widehat{R}(y,0){\rm dv}\Big)\notag\\
 &+\int_{B_{\frac{r}{5}}(x_{0},0)}G(x_{0},y)\widehat{R}(y,0){\rm dv}.
 \end{align}
 Moreover,
\begin{align}\label{R-integral}
\int_{B_{r}(x_{0},0)}G_{r}(x_{0},y)\widehat{R}(y,0){\rm dv} &\leq C(n)\Big( \frac{r^{2}}{\text{vol}(B_{r}(x_{0},0))}\int_{B_{r}(x_{0},1)}\widehat{R}{\rm dv}\notag\\
&+\int_{0}^{r}\frac{s}{\text{vol}(B_{s}(x_{0},0))}\int_{B_{s}(x_{0},0)}\widehat{R}\rm{dvds}\Big),~\forall~r>0.
\end{align}
On the other hand, by the volume comparison  together with  $(\ref{condition-steady})$,
 we have
  \begin{equation}\label{inequality-6}
\frac{1}{\text{vol}(B_{r}(x_{0},0)}\int_{B_{r}(x_{0},0)}\widehat{R}{\rm dv}\leq\frac{C}{\text{vol}(B_{r}(x_{0}))}\int_{B_{r}(x_{0})}R{\rm dv}\leq\frac{\varepsilon_{1}(r)}{1+r},
\end{equation}
where  the function $\varepsilon_{1}(r)\rightarrow0$ as $r\rightarrow\infty$.
 Thus combining (\ref{R-integral}) and (\ref{inequality-6}),  we get from (\ref{green-formula-2}),
\begin{align}
-F(x_{0},t)\leq r^{-2}C(n)t(C^{\prime}t+1)(C^{\prime}t+r\varepsilon_{2}(r))+r\varepsilon_{2}(r) ,\mbox{\quad}\forall \mbox{\ }r>0.\notag
 \end{align}
 where $\varepsilon_{2}(r)\rightarrow0$  as $r\rightarrow\infty$.
Consequently,
\begin{align}\label{r-estimate-4} tR(x_{0},0)\leq r^{-2}C(n)t(C^{\prime}t+1)(C^{\prime}t+r\varepsilon_{2}(r))
+r\varepsilon_{2}(r) ,\mbox{\quad}\forall \mbox{\ }r>0.\end{align}

Now we choose a monotonic $\varepsilon_{3}(r)$ such that $\varepsilon_{3}(r)\rightarrow0$ and $\frac{\varepsilon_{2}(r)}{\varepsilon_{3}(r)}\rightarrow 0$ as $r\rightarrow\infty$. Let $r=t \varepsilon_{3}^{-1}(t)$. Then by (\ref{r-estimate-4} ), we get
$$tR(x_{0},0)\leq C_{1}\varepsilon_{3}^{2}(t)(C^{\prime}t+t\frac{\varepsilon_{2}(t \varepsilon_{3}^{-1}(t))}{\varepsilon_{3}(t \varepsilon_{3}^{-1}(t))})+t\frac{\varepsilon_{2}(t \varepsilon_{3}^{-1}(t))}{\varepsilon_{3}(t \varepsilon_{3}^{-1}(t))} ,\mbox{\quad}\forall \mbox{\ }t\gg1.$$
By dividing  by $t$  on  both sides of the above inequality and then letting  $t\rightarrow\infty$,   it is easy to see that
$R(x_{0},0)=0$.  Since $x_{0}$ is an arbitrary point,  we prove that  $R(x)\equiv0$.
\end{proof}

By Theorem \ref{main-theorem-2} , we can finish the proof of   Theorem \ref{main-theorem-3}.

\begin{proof}[Proof of Theorem \ref{main-theorem-3}~(the steady case)]
Since $(M,g)$ is a  complete non-
\newline compact manifold with nonnegative Ricci curvature, the volume growth of $g$  is at least linear.
Then by (\ref{condition-steady-2}),  it is easy to see that  the average curvature condition (\ref{condition-steady}) is satisfied   in Theorem \ref{main-theorem-2} as in the proof of Lemma \ref{average-decay-1}.   Hence by  Theorem \ref{main-theorem-2},  we  get
  Theorem \ref{main-theorem-3} immediately.
\end{proof}

\section{Boundedness of scalar curvature--II}\label{section-4}

In this section, we  prove the existence and uniqueness of equilibrium point for the steady K\"ahler-Ricci soliton $(M^{n},g,f)$  in Theorem \ref{main-theorem-1}.
As a consequence,   the  maximum of scalar curvature  of $g$ can be attained.

\begin{proof}[Proof of Theorem \ref{main-theorem-1}~(the steady case)] \textbf{Existence:} Let   $g(\cdot,t)=\phi_{t}^{*}g$ be a family of steady solitons generated  by $-\nabla f$. Then  $g(\cdot,t)$ is an eternal solution of   (\ref{kr-flow-2}).  Since $g(\cdot,t)$  has uniformly positive holomorphic bisectional curvature in space time $M\times (-\infty,\infty)$,   we  apply  Theorem 2.1 in \cite{CT5} to  see that there exists a sequence of solutions  $g_\alpha(\cdot, t)=g(\cdot, t_{\alpha}+t)$ on $\Phi_\alpha (D(r))(\subset M)$  such that $\Phi_\alpha^*(g_{\alpha}(\cdot, t))$
converge to a smooth solution  $h(x,t)$ of (\ref{kr-flow-2}) uniformly and smoothly on  a compact subset $D(r)$ for any
 $t\in (-1,\infty)$,    where  $D(r)$ is an Euclidean ball centered at the origin with radius $r$ and $\Phi_{\alpha}$ are local biholomorphisms from $D(r)$ to $M$.      Moreover,  by using the Cao's argument in \cite{Ca1}, it was proved that  $h(x,t)$ is generated by a steady  K\"ahler-Ricci soliton $( D(r), h,f^h)$ with $\nabla f^{h}(o)=0$. Namely,  $h(x,t)$ satisfies
\begin{equation}
R_{i\bar{j}}(h(t))=\nabla_{i}\nabla_{\bar{j}}f^{h}(t),\mbox{\quad}\nabla_{i}\nabla_{j}f^{h}(t)=0,\notag
\end{equation}
where $f^h(t)$ are induced functions  of $f^h$ and $\nabla f^{h}(t)$ vanish  at the origin   for any  $t\in (-1,\infty)$.

     On the other hand,  similar to (\ref{L1-1}),    we have for solitons $(M,g(t),f(t))$,
$$R_{,i}(t)+R_{i\bar{j}}(t)\nabla_{\bar{j}}f(t)=0.$$
Then  $\nabla f(t)$ is determined by the curvature tensor.  Define a  sequence of a family of holomorphic vector fields $V(t_\alpha)$ on $D(r)$  by
$$\Phi_\alpha^*R_{,i}(t_\alpha)+\Phi_\alpha^*R_{i\bar{j}}(t_\alpha)V (t_\alpha)_j=0.$$
 Clearly, $(\Phi_\alpha)_* V (t_\alpha)=\nabla f(t_\alpha)$.  By the convergence of $g_\alpha(\cdot, t)$,    holomorphic vector fields $V (t_\alpha)$ converge to $\nabla f^{h}$ in $C^{\infty}$-topology on $D(r)$ for any $t \in (-1,\infty)$.   Since the eigenvalues of ${\rm Ric}(h(t))$  are positive at $0$  by Proposition 2.2 in \cite{CT5},  the integral curves of $-\nabla^{h}f^{h}$ will converge to $0$ in $D(r)$ when $r$ is  sufficiently small  by the soliton equation. By the convergence of $V (t_\alpha)$, the integral curve of $-V (t_\alpha)$ will also converge to a point $q$ in $D(r_{1})$ for some $r_{1}<r$ when $\alpha$ is large enough (cf. Page 9 of \cite{CT4}). As a  consequence,  $q$ is a zero point of $V (t_\alpha)$  in $D(r_{1})$. This proves that   there exists  a zero point of $\nabla f(t_\alpha)$    in $M$ for each $\alpha$
since $\Phi_{\alpha}^{*}$  is a  local biholomorphism.

\textbf{Uniqueness:} Suppose  that $p$ and $q$ are two equilibrium points. Then $d_{0}(p,q)=d_{t}(p,q)$.
Choose $l>0$ such that $q\in B_{l}(p,0)$. Note that $\phi_{t}:(M^{n},g(t))\rightarrow(M^{n},g(0))$  are a family of  isometric deformations. Thus
$$ C=\inf_{x\in B_{l}(p, t)}\mu_{1}(x,t)= \inf_{x\in B_{l}(p, 0)}\mu_{1}(x,0)>0,\mbox{\ }\forall x\in  B_{l}(p, 0),$$
where $\mu_{1}(x,t)$ is the smallest eigenvalue of $\text{Ric}(x,t)$ w.r.t $g(x,t)$. Since the metric is decreasing along the flow, we see that  $B_{l}(p, 0)\subset B_{l}(p, t)$.
 Hence  by   (\ref{ricci-flow-2}),   we get
\begin{align}
\frac{{\rm d}|v_{x}|^{2}_{t}}{\rm{dt}}\leq-\mu_{1}(x,t)|v_{x}|^{2}_{t}\leq -C|v_{x}|^{2}_{t}, ~\forall~ t\geq0,\notag
\end{align}
 where $x\in B_{l}(p, 0)$ and $v_{x}\in T_{x}^{(1,0)}M$.   Therefore,  if  we let  $\gamma(s)$ be   a minimal geodesic curve connecting  $p$ and $q$    with  an arc-parameter $s$  w.r.t the metric $g(x,0)$   in  $B_{l}(p,0)$,    we deduce
\begin{align}
 d_{t}(p,q)\leq\int_{0}^{d}|\gamma^{\prime}(s)|_{t}{\rm ds}\leq\int_{0}^{d}|\gamma^{\prime}(s)|_{0} e^{-Ct}{\rm ds}=d_{0}(p,q)e^{-Ct}.\notag
\end{align}
Letting $t\to \infty$,  we  see that  $d_{t}(p,q)=d_{0}(p,q)=0$.   This proves that  $p=q.$
\end{proof}

\begin{prop}\label{proposition-5.7} Let  $(M^{n},g, f)$ be a simply connected complete  non-comp-act steady gradient K\"{a}hler-Ricci soliton  with  nonnegative bisectional curvature.
 Suppose that $g$ satisfies
\begin{align}\label{curvature-condition-3}
 R(x)\leq \frac{C}{ 1+r(x)^{1+\epsilon}},
\end{align}
for some $\epsilon >0, C$. Then $(M,g)$ is isometric to $\mathbb C^n$.

\end{prop}

 \begin{proof}
We suffice  to show  that $g$  is Ricci flat.  On the contrary,    we  may assume that the Ricci curvature  of $g(\cdot, t)$  is  positive everywhere  by  a dimension   reduction theorem of Cao  for K\"ahler-Ricci flow on  a simply connected complete  K\"ahler manifold with  nonnegative bisectional curvature \cite{Ca2}, where $g(\cdot,t)=\phi_{t}^{*}g$ is   the generated  solution of   (\ref{kr-flow-2})   as in Section 4.  Let  $o$ be the  unique equilibrium point
 of $g$  according to  Theorem \ref{main-theorem-1}.  In the following we  use an argument of Hamilton in \cite{H1}  to show  that there exists a pointedwise backward limit $g_{\infty}(x)$  of $g(x,t)$
  on $M\setminus\{o\}$ and $(M\setminus\{o\},g_{\infty})$ is a complete flat Riemannian manifold.

 Since
   $$R(x)+|\nabla f|^{2}=R(o), $$
 by (\ref{curvature-condition-3}),  we see
 $$\lim_{d(x,o)\rightarrow\infty}|\nabla f|^{2}(x)=R(o)>0.$$
Note the equilibrium point  is unique. It follows
$$C_\delta=\inf_{M\setminus B_{\delta}(o)}|\nabla f|^{2}>0,\mbox{\quad}\forall \mbox{\ }\delta>0.$$
 This implies
 \begin{equation}
 d(\phi_{t}(x),o)\geq C_\delta|t|\mbox{\quad} \forall \mbox{\ }x\in M\setminus B_{\delta}(o)\mbox{,\ }t\leq 0 .\notag
 \end{equation}
 Hence by (\ref{curvature-condition-3}),    we get  from equation (\ref{kr-flow-2}),
 \begin{align}0&\leq -\frac{\partial}{\partial t}g(x,t)\leq  R (g(x,t)) g(x.t)\notag\\
 &\le\frac{C}{1+d^{1+\epsilon}(\phi_{t}(x),o)}g(x,t)\le \frac{C_\delta'}{1+ |t|^{1+\epsilon}} g(x,t).
 \notag
 \end{align}
Therefore,  we derive
 \begin{equation}\label{T5-3}
 g(x,0)\leq g(x,t_{1})\leq g(x,t_{2})\leq C_\delta' g(x,0),
 \end{equation}
 for any  $x\in M\setminus B_{\delta}(o,0)$ and $-\infty < t_{2}\leq t_{1}\leq 0$.

 By ($\ref{T5-3}$) and Shi's higher order estimate for curvatures,  we see that
    $g(x,t)$  converge locally to a limit  K\"ahler metric  $g_{\infty}(x)$ on $M\setminus\{o\}$ as $t\to -\infty$.   Clearly,  $g_{\infty}(x)$ is  Ricci-flat since
    $$  0=\lim_{t\to-\infty}  -\frac{\partial}{\partial t}g(x,t)=\lim_{t\to-\infty}\text{Ric}(g(\cdot,t))=\text{Ric}(g_\infty).$$
     Consequently,  $g_\infty$ is flat.   Moreover, $g_{\infty}$ is a complete, because
\begin{align}\lim_{x'\to o}d_{g_\infty}(x,x') =\lim_{t\rightarrow-\infty}d_{g(\cdot,t)}(x,o)=\lim_{t\rightarrow-\infty}d_g(\phi_{t}(x),o)=\infty,\notag
\end{align}
where $x\in M\setminus \{o\}$.  On the other hand,  it was proved by Chau and Tam that  $M$ is biholomorphic to $\mathbb{C}^{n}$ since $M$ is a
a simply connected complete  non-compact steady gradient K\"{a}hler-Ricci soliton  with positive Ricci curvature \cite{CT6}.  Thus, $M\setminus \{o\}$ is simply connected. Hence, $M\setminus \{o\}$ is also biholomorhic to $\mathbb{C}^{n}$.   This is a contradiction!  Therefore,  $g$ is Ricci flat and consequently,
$(M,g)$ is isometric to $\mathbb C^n$.

 \end{proof}

\end{document}